\documentclass[12pt]{amsart}
\usepackage[colorlinks=true,pagebackref,hyperindex,citecolor=blue,linkcolor=red]{hyperref}
\usepackage{amsmath}
\usepackage{amsfonts}
\usepackage{setspace}
\usepackage{moreverb}
\usepackage[dvips]{graphicx}
\usepackage[latin1]{inputenc}
\usepackage[T1]{fontenc}
\usepackage{amssymb}
\usepackage{tikz}
\usepackage{color}
\usepackage[all]{xy}  
\usepackage{xfrac}
\usepackage[top=1in, bottom=1in, left=1in, right=1in]{geometry}
\usepackage{mathrsfs}
\newcommand{\cb}{ }

\newtheorem{theoremx}{Theorem}

\theoremstyle{definition}
\newtheorem{theorem}{Theorem}[section]

\newtheorem{corollary}[theorem]{Corollary}
\newtheorem{lemma}[theorem]{Lemma}
\newtheorem{proposition}[theorem]{Proposition}

\theoremstyle{definition}
\newtheorem{definition}[theorem]{Definition}

\newtheorem{example}[theorem]{Example}

\newtheorem{remark}[theorem]{Remark}
\numberwithin{equation}{subsection}

\newcommand{\NN}{\mathbb{N}}
\newcommand{\ZZ}{\mathbb{Z}}
\newcommand{\QQ}{\mathbb{Q}}
\newcommand{\FF}{\mathbb{F}}

\newcommand{\m}{\mathfrak{m}}
\newcommand{\n}{\mathfrak{n}}
\newcommand{\fa}{\mathfrak{a}}
\newcommand{\fb}{\mathfrak{b}}

\newcommand{\ls}{\leqslant}

\newcommand{\cK}{\mathcal{K}}

\newcommand{\cA}{\mathcal{A}}
\newcommand{\cJ}{\mathcal{J}}

\newcommand{\ehk}{\e_{HK}}

\newcommand{\pd}{\operatorname{pd}}
\newcommand{\cd}{\operatorname{cd}}

\newcommand{\e}{\operatorname{e}}

\newcommand{\Ht}{\operatorname{ht}}

\newcommand{\Cech}{ \check{\rm{C}}}

\newcommand{\Depth}{\operatorname{depth}}	
\newcommand{\Char}{\operatorname{char}}

\newcommand{\wS}{\widetilde{S}}

\newcommand{\fp}{\mathfrak{p}}

\begin{document}
\newcommand{\tens}{\otimes}
\newcommand{\hhtest}[1]{\tau ( #1 )}
\renewcommand{\hom}[3]{\operatorname{Hom}_{#1} ( #2, #3 )}

\title{Graph Connectivity and Binomial Edge Ideals}
\author{Arindam Banerjee}
\author{Luis N\'u\~nez-Betancourt}
\maketitle

\begin{abstract}
We relate homological properties of a binomial edge ideal $\mathcal{J}_G$ to invariants that measure the connectivity of a simple graph $G$. Specifically, 
we show  if  $R/\mathcal{J}_G$ is a Cohen-Macaulay ring, then graph toughness of $G$ is exactly $\frac{1}{2}$.
We also give an inequality between the depth of $R/\mathcal{J}_G$ and the vertex-connectivity of $G$. 
In addition, we study the Hilbert-Samuel multiplicity, and the Hilbert-Kunz multiplicity of $R/\mathcal{J}_G$.

\end{abstract}

\renewcommand{\thefootnote}{\fnsymbol{footnote}} 
\footnotetext{{\bf{MSC Classes:}} 13C14,	05C40, 	05E40.}     
\footnotetext{{\bf{Key Words:}} Binomial Edge ideals, Graph Toughness, Vertex Connectivity, \& Cohen-Macaulay Rings. }     
\renewcommand{\thefootnote}{\arabic{footnote}}

\section{Introduction}

The interplay between the combinatorics of
a graph and the algebra of various ideals arising from it has been a vibrant area of research in the last two decades. The most studied ideal in this context is  the edge ideal. The interplay of 
Castelnuovo-Mumford regularity, projective dimension, and  Cohen-Macaulayness  of an edge
ideal with the combinatorics of the underlying graph has been studied
extensively. Two important examples of this research are 
Fr\"{o}berg's \cite{Fro} characterization of edge ideals with linear minimal free resolutions and Herzog and
Hibi's \cite{HH2} characterization of Cohen-Macaulay bipartite edge ideals.

The binomial edge ideal $\cJ_G$ of a simple graph $G$ was introduced independently  by Herzog, Hibi,  Hreinsd\'ottir, Kahle, and Rauh  \cite{BinomialEdgeIdeal}, and  Ohtani \cite{Ohtani}. 
 This ideal arises naturally in the study of  conditional independence ideals that are suitable to model robustness in the context of algebraic statistics \cite{BinomialEdgeIdeal}. Various aspects of the homological algebra of the binomial edge ideals have been related with the combinatorial structure of the underlying graph. For instance, the regularity of $\cJ_G$ is bounded below by the length of the largest induced path of $G$ and above by the number of vertices of $G$ \cite{MatsudaMurai}. 
In addition, Cohen-Macaulayness has been characterized for binomial edge ideals of chordal graphs in terms of the maximal cliques of the underlying graph \cite{ChordalBEI}. Furthermore, it has been proved that $\cJ_G$ is a complete intersection if and only if $G$ is a disjoint union of paths \cite{ChordalBEI,Rinaldo}. 

In this article we study the relationship of  a binomial edge ideal and the connectivity of the graph.  Specifically, we investigate interactions between toughness and  vertex connectivity of a graph with dimension and depth of $R/\cJ_G.$
Graph toughness, $\tau(G)$, of a graph was first introduced by Chv\'atal \cite{Chvatal}, and serves as a measure of connectivity for a graph (see Section \ref{SectionToughness}). This invariant can be seen as the minimum ratio of the cardinality of a set of vertices that makes $G$ disconnected by removing them, and the number of connected components in the induced graph. The relation between toughness and Hamiltonicity has been a topic of much research interest \cite{SurveyToughness}.
The vertex connectivity gives the minimum number of vertices necessary to delete from $G$ to obtain a disconnected graph.
These measures of connectivity relate in the following way:  high toughness implies high vertex connectivity (see Remark \ref{Rem both conn}). 

Our first main theorem gives a necessary condition on the toughness of $G$ for Cohen-Macaulay rings given by binomial edge ideals.

\begin{theoremx}[see Theorem \ref{Main Tech}]\label{Main}
Let $G$ be a connected graph on $[n]$  and $\cJ_G$ the corresponding binomial edge ideal in $R=K[x_1,\ldots,x_n,y_1,\ldots,y_n]$. If $R/\cJ_G$ Cohen-Macaulay, then either  $\tau(G)=\frac{1}{2}$ or $G$ is the complete graph. 
\end{theoremx} 

The proof of Theorem \ref{Main} has two main ingredients. The first is a relation between the toughness of $G$ and the dimension of $R/\cJ_G$ (see Theorem \ref{Thm Dim}). This relation was motivated from the fact the dimension of $R/\cJ_G$ is given by the difference of the number of vertices removed and the number of connected components in the induced graph, while the toughness is given by its ratio. The second ingredient is the fact that if the quotient ring of  $\cJ_G$ is $S_2$, then $G$ cannot be $2$-vertex-connected (see Proposition \ref{Prop 2-vertex-conec}). Pushing this idea, we obtain our second main theorem, which establishes a stronger relation between the  depth  of $R/\cJ_G$ in terms of  vertex connectivity of $G$.

\begin{theoremx}[{see Theorem \ref{Main Depth Char p} \& \ref{Main Depth Char Zero}}]  

\label{Main Depth}
Let $G$ be a connected graph on $[n]$ and $\cJ_G$ the corresponding binomial edge ideal in $R=K[x_1,\ldots,x_n,y_1,\ldots,y_n]$. 
Suppose that $G$ is not the complete graph and that $\ell$ is the vertex connectivity of $G$. Then, 
$$\Depth(R/\cJ_G)\leq n-\ell+2\hbox{ and }\pd(R/\cJ_G)\geq n+\ell-2.$$
\end{theoremx} 
 
 Theorem \ref{Main} and \ref{Main Depth} establish that if $R/\cJ_G$ is Cohen-Macaulay, then the toughness is $\frac{1}{2}$ and the vertex-connectivity $1$. Therefore, we give  new combinatorial criteria to identify graphs such that $R/\cJ_G$ is not Cohen-Macaulay.

In the final section of this paper, we study invariants of $\cJ_G$.
We take advantage of the fact that binomial edge ideals can be viewed as a generalization of the ideals of maximal minors of $2 \times n$ generic matrices. 
Specifically, those ideals are the binomial edge ideals of complete graphs on $n$ vertices.
The ideals of maximal minors are very well studied objects with many nice properties \cite{Bruns}; for instance, they are Cohen-Macaulay and there are formulas for several invariants associated to them. 

We use the properties of determinantal ideals to give a formula for the Hilbert-Kunz and Hilbert-Samuel multiplicities for binomial edge ideals in terms of the combinatorics of the graph (see Theorem \ref{Thm Mult}).
As a consequence of these results, the Hilbert-Kunz multiplicity of $R/\cJ_G$ is a rational number.  Furthermore, the values of the multiplicities depend only on the structure of the graph and not on the characteristic of the coefficient field. For instance, all the binomial edge ideals associated to a Hamiltonian graph with $n$ vertices have the same multiplicities.

\section{Background}
In this section we establish terminology and recall basic results from graph theory. In addition, we recall the definition and first properties of binomial edge ideals. We make use of these preliminaries throughout all the manuscript.

\subsection{Graph terminology}\label{SectionToughness}
In this subsection we review basic definitions and results from graph theory. We refer to \cite{Bollobas} for an introduction to this theory.
In this manuscript we only consider simple undirected graphs, whose vertices belong to $[n]:=\{1,\ldots, n\}$.
Given a graph $G$ and a subset of vertices $S$,  $c(S)$ denotes the number of connected components of $G\setminus S.$

\begin{definition}
We say that $G$ is \emph{$\ell$-vertex-connected} if $\ell< n$ and for every subset $S$ of vertices such that $|S|<\ell$, the induced graph $G\setminus S$ is connected. The \emph{ vertex-connectivity of $G$, {\cb denoted by $\kappa(G)$,}} is defined as the maximum integer $\ell$ such that $G$ is $\ell$-vertex-connected. 
\end{definition}

We now present definitions and properties related to graph toughness. We refer to \cite{SurveyToughness} for a survey on this topic.

\begin{definition}[{\cite{Chvatal}}]
We say that a connected graph is \emph{$t$-tough} if for every subset $S\neq\varnothing$ such that $c(S)\geq 2$, we have $t\cdot c(S)\leq |S|.$
We define the toughness of $G$, $\tau(G),$ as the maximum value of $t$ for which $G$ is $t$-tough.
\end{definition}

\begin{remark}\label{Toughness Frac}
If a graph, $G$, is not complete, then 
$$
\tau(G)=\min\left\{\frac{|S|}{c(S)}: c(S)\geq 2\right\}.
$$
\end{remark}

\begin{example}
The following is a list of the toughness of certain classes of graphs:
\begin{enumerate} 
\item If $\cK_n$ is the complete graph, then $\tau(\cK_n)=\infty$ {\cb  and $\kappa(G)=n-1$}.
\item If $K_{m,n}$ is the complete bipartite graph with $2\leq m\leq n,$ then $\tau(\cK_{m,n})=\frac{m}{n}$ {\cb and $\kappa(G)=m$}. 
\item If $P_n$ is the $n$-path, with $3\leq n$,  then $\tau(P_n)=\frac{1}{2}$ {\cb and $\kappa(G)=1$}.
\item If $C_n$ is the $n$-cycle, with $4\leq n,$ then $\tau(C_n)=1$ {\cb  and
$\kappa(G)=2.$}
\item If $G$ is a Hamiltonian graph, then $\tau(G)\geq 1$ {\cb  and
$\kappa(G)\geq 2.$}
\end{enumerate}
\end{example}

{\cb
\begin{remark}\label{Rem both conn}
Suppose that $G$ is not the complete graph.
If $G$  is $t$-tough, with $t>0,$ then $ \frac{|S|}{c(S)}\geq t$ for every subset $S\subseteq [n]$ such that $c(S)\geq 2.$ Then, 
$|S|\geq t\cdot c(s)  $ for every set such that $c(S)\geq 2.$ Then, $2t\geq |S|$ is a necessary condition to have $c(S)\geq 2.$ This means that if $k<2t$ vertices are removed, then $c(S)=1,$ so $G\setminus S$ remains connected. Hence, $G$ is a $\lceil 2t\rceil$-vertex-connected graph.
\end{remark}
}
\subsection{Binomial edge ideals}
Throughout this paper $R$ denotes $K[x_1,\ldots, x_n,y_1,\ldots, y_n]$, a polynomial ring in $2n$ variables over a field $K$.

\begin{definition}
Let $G$ be a graph on $[n]$. We define the \emph{binomial edge ideal corresponding to $G$} by
$$
\cJ_G=(x_iy_j-x_jy_i: \{i,j\}\in G \hbox{ \& }i\neq j).
$$
\end{definition}

\begin{example}
Let $G$ be the complete graph on  $[n]$. The binomial edge ideal corresponding to $G$ is
$$
\cJ_G=(x_iy_j-x_jy_i: i\neq j)=I_2(X),
$$
where $X=\begin{pmatrix}
x_1 &\ldots & x_n\\
y_1 & \ldots & y_n
\end{pmatrix}.$
\end{example}

\begin{definition}
Let $G$ be a graph on $[n],$ and $S\subseteq [n]$.
Let $G_1,\ldots, G_{c(S)}$  denote the connected components of $G\setminus S.$ Let $\widetilde{G}_i$ denote the complete graph on the vertices of $G_i.$ We set
$$
P_S(G)=\left( \bigcup_{i\in S}\{x_i,y_i\}, \cJ_{\widetilde{G}_1},\ldots, \cJ_{\widetilde{G}_{c(S)}}\right)R.$$
\end{definition}
It is well-known that $P_S(G)$ is a prime ideal for every $S\subseteq [n]$.
Furthermore, these prime ideals play a very important role to study $\cJ_G.$

\begin{theorem}[{\cite[Theorem 3.2]{BinomialEdgeIdeal}}]\label{IntersectionPrimes}
Let $G$ be a graph on $[n]$, and $\cJ_G$ its binomial edge ideal in $R$.
Then, $$\cJ_G=\bigcap_{S\subseteq [n]} P_S(G).$$
\end{theorem}

\begin{proposition}[{\cite[Corollary 3.9]{BinomialEdgeIdeal}}]\label{MinorMinimal}
If $G$ is a connected graph on $[n]$, then
$P_\varnothing (G)$ is a minimal prime of $\cJ_G.$
\end{proposition}

\begin{remark}\label{RemDim}
From Theorem \ref{IntersectionPrimes}, we obtain 
$$
\dim R/\cJ_G=\max\{n-c(S)+|S|: S\subseteq [n]\}
$$
because $\dim R/P_S(G)=n-c(S)+|S|$ \cite[Corollary 3.3]{BinomialEdgeIdeal}. 
If $G$ is connected and $S\subsetneq [n]$ is such that $c(S)=1,$ then 
$P_\varnothing(G)\subseteq P_S(G).$ Therefore, $P_S(G)$ cannot be a minimal prime. 
Hence,
$$\cJ(S)=\left(\bigcap_{c(S)\geq 2}P_S(G)\right)\bigcap P_\varnothing (G)$$ and 
$$
\dim R/\cJ_G=\max\{n+c(S)-|S|: c(S)\geq 2 \hbox{ or }S=\varnothing\}
$$
\end{remark}

\section{Graph connectivity and homological properties of binomial edge ideals}
In this section we show Theorems \ref{Main} and \ref{Main Depth}, which are the main results of this paper.

\subsection{Graph Toughness and  binomial edge ideals}

In this subsection we establish relations between graph toughness, and different aspects related to the dimension and depth of a binomial edge ideal. These relations are necessary ingredients to prove Theorem \ref{Main}. 
We start by observing restrictions for $1$-tough graphs.

\begin{proposition}\label{equiv 1-tough}
Let $G$ be a connected graph on $[n]$ and $\cJ_G$ the corresponding binomial edge ideal in $R.$ 
Then, the following are equivalent:
\begin{enumerate}
\item $G$ is $1$-tough;
\item $\dim(R/\cJ_G)=n+1$ and $P_\varnothing(G)$ is the only minimal prime of dimension  $n+1.$
\end{enumerate}
\end{proposition}
\begin{proof}
A graph is $1$-tough if and only if $c(S)\leq |S|$ for every $S$ such that $c(S)\geq 2.$
 This happens if and only
$\dim P_S(G)=n+c(s)-|S|\leq n$ for 
$c(S)\geq 2.$ 
Then, the result follows from Remark \ref{RemDim}.
\end{proof}

As an immediate consequence of Proposition \ref{equiv 1-tough}, we obtain a characterization binomial edge ideals that are equidimensional and come from $1$-tough graphs.

\begin{remark}\label{Rem 1-tough equi}
Let $G$ be a connected graph on $[n]$ and $\cJ_G$ the corresponding binomial edge ideal in $R.$ Suppose that $G$ is $1$-tough.
Then, the following are equivalent:
\begin{enumerate}
\item $R/\cJ_G$ is a Cohen-Macaulay ring;
\item $R/\cJ_G$ is an equidimensional ring;
\item $R/\cJ_G$ is a domain;
\item $G$ is the complete graph.
\end{enumerate}
\end{remark}

We now present a theorem that gives bounds on the dimension of $R/\cJ_G$ and the toughness of $G.$ As a consequence, the equidimensionality of $R/\cJ_G$ puts strong restrictions on the  toughness of $G$. This plays a key role in the proof of Theorem \ref{Main}.

\begin{theorem}\label{Thm Dim}
Let $G$ be a connected graph on $[n]$ and $\cJ_G$ the corresponding binomial edge ideal in $R.$ 
Suppose that $\tau(G)<1.$
Then,
$$n+\frac{1-\tau(G)}{\tau(G)}\leq \dim(R/\cJ_G)\ls n+(n-1)\cdot (1-\tau(G)).$$
\end{theorem}
\begin{proof}
Since $\tau(G)<1$, $G$ is not a complete graph. Then, $$
\tau(G)=\min\left\{\frac{|S|}{c(S)}: c(S)\geq 2\right\}
$$
by Remark \ref{Toughness Frac}.

On one hand, we have
$\tau(G)\leq \frac{|S|}{c(S)}$
for every  $S\subsetneq [n]$ such that $c(S)\geq 2.$ 
Then, $\tau(G)  \cdot c(S)\leq |S|,$ and so,
$c(S)-|S|\leq (1-\tau(G)) \cdot c(S).$
We note that $ c(S)\leq n-1$ for $S\subsetneq[n].$ Then,
\begin{equation}
\label{eq c-s}
c(S)-|S|\leq (n-1)\cdot (1-\tau(G)).
\end{equation}
for every $S$ such that $c(S)\geq 2.$
Therefore, 
$$\max\{ n+c(S)-|S|: c(S)\geq 2\} \leq n+(n-1)\cdot (1-\tau(G))$$ 
by \ref{eq c-s}.

Since $\tau(G)<1,$ we have  $\tau(G)\leq 1-\frac{1}{n-1}$ from {\cb Remark \ref{Toughness Frac}} because $c(S)\leq n-1$.
Then, $$n+1\leq n+(n-1)\cdot (1-\tau(G)).$$ 
Hence, 
$$
\dim(R/\cJ_G)  =\max\{n+c(S)-|S|: c(S)\geq 2 \hbox{ or }S=\varnothing \} \leq n+(n-1)\cdot (1-\tau(G))
$$
by Remark \ref{RemDim}.

On the other hand, there exists $ \wS\subsetneq [n]$ such that $c(\wS)\geq 2$ and $\tau(G)=\frac{|\wS|}{c(\wS)}$ by Remark \ref{Toughness Frac}.
Then, $ c(\wS)=\frac{1}{\tau(G)}|\wS|,$ and so, $ |\wS|\left(\frac{1-\tau(G)}{\tau(G)}\right)= c(\wS)-|\wS|.$ 
Since $|\wS|\geq 1,$ we have 
$\frac{1-\tau(G)}{\tau(G)} \leq c(\wS)-|\wS|.$ 
Therefore, 
$$n+\frac{1-\tau(G)}{\tau(G)} \leq n+c(\wS)-|\wS|\leq \dim(R/\cJ_G)$$
by Theorem \ref{IntersectionPrimes} (see also Remark \ref{RemDim}).
\end{proof}

As a consequence of the previous theorem, we have the dimension of $R/\cJ_G$ imposes restrictions on the toughness of $G$.

\begin{corollary}\label{Cor tau G}
Let $G$ be a connected graph on $[n]$ and $\cJ_G$ the corresponding binomial edge ideal in $R.$ 
Then,
$$
\frac{1}{\dim(R/\cJ_C)-n+1}\leq \tau(G).$$
In addition, if $G$ is not $1$-tough, then
$$
\tau(G)\leq \frac{\dim(R/\cJ_G)-1}{n-1}.
$$
\end{corollary}
\begin{proof}
Since $\dim(R/\cJ_C)\geq n+1,$ we have  $\frac{1}{\dim(R/\cJ_C)-n+1}\leq \frac{1}{2}$. 
{\cb We note that if $G$ is not $1$-tough, then $\tau(G)<1.$}
Then, it suffices to show the statements for $\tau(G)<1.$
The inequalities follow directly from Theorem \ref{Thm Dim}.
\end{proof}

We now get a corollary that plays an important role in Theorem \ref{Main}.
In particular, the following result shows that if $R/\cJ_G$ is Cohen-Macaulay, then $\frac{1}{2}\leq \tau(G)$.

\begin{corollary}\label{Tough Equi}
Let $G$ be a connected graph on $[n]$ and $\cJ_G$ the corresponding binomial edge ideal in $R.$
If $R/\cJ_G$ is equidimensional,  either $G$ is the complete graph or $\frac{1}{2}\leq \tau(G)<1.$ 
\end{corollary}
\begin{proof}
We assume that $G$ is not the complete graph.
From Remark \ref{Rem 1-tough equi}, a $1$-tough graph is equidimensional if and only if it is complete. 
Then, $\tau(G)<1.$
If $\cJ_G$ is equidimensional, then $\dim(R/\cJ_G)=n+1$ by Theorem \ref{IntersectionPrimes} and Proposition \ref{MinorMinimal}. Then, $\frac{1}{2}\leq \tau(G)$ by Corollary \ref{Cor tau G}.
\end{proof}

We now start with preparation results needed to show the other inequality of Theorem \ref{Main}: if $R/\cJ_G$ is Cohen-Macaulay, then $ \tau(G)\leq  \frac{1}{2}$.

\begin{lemma}\label{lemma dim vertex-connectivity}
Let $G$ be a connected graph on $[n]$ and $\cJ_G$ the corresponding binomial edge ideal in $R.$
Suppose that $G$ is $\ell$-vertex-connected, $\ell\geq 1$, $G$ is not the complete graph, and $\dim R/\cJ_G=n+1$.
Then, $\dim(R/(P_\varnothing(G)+Q))\leq n-\ell+1$ for every minimal prime, $Q\neq P_\varnothing(G)$, of $\cJ_G$.
\end{lemma}
\begin{proof}
We note that $\cJ_G$ is not a prime ideal because $G$ is not the complete graph.
We consider two cases. 
Suppose that $\ell=1.$ Since $G$ is connected, we have  $\dim R/P_\varnothing (G)=n+1.$ Then, for any minimal prime $Q\neq P_\varnothing(G),$ we have $\dim R/P_\varnothing (G)\geq n.$

We now assume that $\ell\geq 2.$
Let $S\subseteq [n]$ with {\cb $0<|S|< \ell.$}
Since $G$ is $\ell$-vertex-connected, $P_S(G)$ contains $P_{\varnothing}(G)$ (see \cite[Corollary 3.9]{BinomialEdgeIdeal}).
Then, $P_S(G)$ cannot be a minimal prime if $|S|<\ell.$

We now assume that $Q=P_S(G)$ is a minimal prime, and so,  $|S|\geq \ell.$
Then, $P_{\varnothing}(G)+Q$ is generated by
$$
\{ x_iy_j-x_j y_i: i,j\in [n]\setminus S\}\cup\{x_i,y_i:i\in S \}.
$$ 
Then, $\dim R/(P_{\varnothing}(G)+Q)=n-|S|+1\leq n-\ell+1.$
\end{proof}

We now focus on $2$-vertex-connected graphs. We first need to recall the definition and a few properties of local cohomology. We refer to \cite{BroSharp} for an introduction to this cohomological theory.

\begin{definition}\label{Def LC}
Let $I\subseteq R$ be a  ideal  generated by polynomials $\underline{f}=f_1,\ldots,f_\ell\in R$, and $M$ an $R$-module. The \emph{$\check{\mbox{C}}$ech complex of $M$ with respect to the sequence $\underline{f}$, $\Cech^\bullet(\underline{f};M)$,} is defined by
{\cb
$$
0\to M\to \bigoplus_i M_{f_i}\to\bigoplus_{i,j} M_{f_i f_j}\to \ldots \to M_{f_1 \cdots f_\ell} \to 0,
$$
}
where $\Cech^i(\underline{f};M)=\bigoplus \limits_{1 \leq j_1<\ldots<j_i\leq \ell} M_{f_{j_1}\cdots f_{j_i}}$ and the morphism 
in every summand is a localization map up to sign.
\emph{The local cohomology of $M$ with support on $I$} is defined by
$$
H^i_I(M)=H^i(\Cech^\bullet(\underline{f};M)).
$$
\end{definition}

\begin{definition}
Let $I\subseteq R$ be an ideal. We define \emph{the cohomological dimension of $I$} by
$$
\cd(I)=\max\{i\in\NN: H^i_I(R)\neq 0\}.
$$
\end{definition}

Before we are finally ready to prove  Theorem \ref{Main}, we need to recall the definition of Serre's conditions.

\begin{definition}
Let $T$ be a Noetherian ring. We say that $T$ is an $S_\ell$ ring if 
$$
\Depth(T_\fp)=\min\{\ell, \dim T_\fp \}
$$
for every prime ideal $\fp.$ In particular, $T$ is Cohen-Macaulay if and only if it is an $S_{\dim(T)}$ ring.
\end{definition}

The following result is our final ingredient in the proof of Theorem \ref{Main}. 

\begin{proposition}\label{Prop 2-vertex-conec}
Let $G$ be a connected graph on $[n]$ and $\cJ_G$ the corresponding binomial edge ideal in $R.$
Suppose that $G$ is $2$-vertex-connected,  and that $G$ is not the complete graph.
Then, $R/\cJ_G$ is not an $S_2$ ring. In particular, if $R/\cJ_G$ is Cohen-Macaulay, then $G$ is not a $2$-vertex-connected graph.
\end{proposition}
\begin{proof}
We proceed by contradiction and assume that $R/\cJ_G$ is an $S_2$ ring.
Since $R/\cJ_G$ is the quotient of a polynomial ring over a homogeneous ideal, we have $R/\cJ_G$ is an equidimensional ring of dimension $n+1$.
Since $G$ is not the complete graph, $\cJ_G$ is not a prime ideal.
Let $P_{\varnothing}(G),Q_1,\ldots, Q_t,$ be the minimal primes of $\cJ_G$. In particular, $\cJ_G=P_{\varnothing}(G)\cap Q_1\cap\ldots\cap Q_t$ because the binomial edge ideals are radical.
Let $I=Q_1 \cap\ldots\cap Q_t$.
For every prime ideal $\fp$ such that $P_{\varnothing}(G)+I\subseteq \fp,$
we have $Q_i\subseteq \fp$ for some $i$. Then, $P_{\varnothing}(G)+Q_i\subseteq \fp,$ and $\dim(R/P_{\varnothing}(G)+I)\leq \dim(R/(P_{\varnothing}(G)+Q_i)\leq n-1$ by Lemma \ref{lemma dim vertex-connectivity}.

From the short exact sequence
$$
0\to R/\cJ_G\to R/P_{\varnothing}(G)\oplus R/I\to R/(P_{\varnothing}(G)+I)\to 0,
$$
we obtain an associated long exact sequence 
$$
\ldots\to H^{n+1}_{\m} (R/\cJ_G)   \to
 H^{n+1}_{\m}(R/P_{\varnothing}(G))\oplus H^{n+1}_{\m}(R/I)\to  H^{n+1}_{\m}(R/(P_{\varnothing}(G)+I))\to 0.
$$

Since $$\dim(R/P_{\varnothing}(G)+I)\leq \dim(R/(P_{\varnothing}(G)+Q_i)\leq n-1,$$ we deduce that $ H^{n}_{\m}(R/(P_{\varnothing}(G)+I)= H^{n+1}_{\m}(R/(P_{\varnothing}(G)+I)=0$ by Grothendieck's Vanishing Theorem.
Then, $H^{n+1}_{\m} (R/\cJ_G)\cong H^{n+1}_{\m}(R/P_{\varnothing}(G))\oplus H^{n+1}_{\m}(R/I)$. Since $R/\cJ_G$ is equidimensional, $H^{n+1}_{\m}(R/I)\neq 0$ by Grothendieck's Non-vanishing Theorem. Then, $H^{n+1}_{\m} (R/\cJ_G)$ decomposes, and so, $R/\cJ_G$ cannot be $S_2$ \cite[Corollary 3.7]{HHgraph}.
\end{proof}

We now prove a more  general version of Theorem \ref{Main}, which only assumes that $R/\cJ_G$ is $S_2$.

\begin{theorem}\label{Main Tech}
Let $G$ be a connected graph on $[n]$ and $\cJ_G$ the corresponding binomial edge ideal in $R.$
If $R/\cJ_G$ is  $S_2$, then either $G$ is the complete graph or $\tau(G)=\frac{1}{2}$. In particular, this holds when $R/\cJ_G)$ is Cohen-Macaulay.
\end{theorem}
\begin{proof}
We suppose that $G$ is not the complete graph.
By Corollary \ref{Tough Equi},  $\frac{1}{2}\leq\tau(G).$
By Proposition \ref{Prop 2-vertex-conec}, $G$ cannot  be $2$-vertex-connected
because $\dim(R/\cJ_G)=n+1$ and $R/\cJ_G$ is $S_2$.
There exists a subset $S\subseteq [n]$ with one element such that $c(S)\geq 2.$
Then, $$\tau(G)=\min\left\{\frac{|S|}{c(S)}:  c(S)\geq 2\right\}\leq \frac{1}{2}.$$
The last claim follows from the fact that a standard graded Cohen-Macaulay $K$-algebra is equidimensional and $S_2.$
\end{proof}

\begin{remark}
We point out that the converse of the previous theorem is not true. For instance, if $t>2$, the h $\cJ_{\cK_{t,2t}}$ is not even equidimensional \cite[Lemma 3.2]{SZBipartite}, but its toughness is $\frac{1}{2}.$
\end{remark}

{\cb
We give an example that was kindly suggested by the referee.
\begin{example}
The  toughness of the following graph is $\frac{1}{2}$, and its vertex-connectivity is $1$. However, $R/\cJ_G$ is not Cohen-Macaulay, because its  depth is $6$ and its dimension is $7$. This computation was verified using Macaulay2 \cite{M2}.
\begin{center}
\includegraphics[scale=.25]{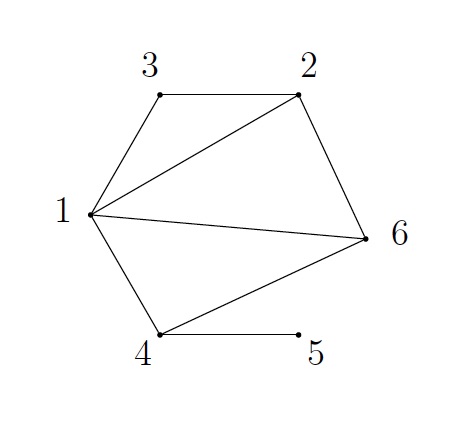}
\end{center}
\end{example}
}
\begin{example}
Suppose that $G$ is chordal graph on $[n]$ satisfying:
\begin{enumerate}
\item any two distinct maximal cliques intersect in at most one vertex;
\item each vertex of G is the intersection of at most two maximal cliques.
\end{enumerate}
 Then, $R/\cJ_G$ is Cohen-Macaulay \cite[Theorem 1.1]{ChordalBEI}. Hence, $\tau(G)=\frac{1}{2}$ by Theorem \ref{Main Tech}.
\end{example}

\begin{corollary}
Let $G$ be a graph  $[n]$ and $\cJ_G$ its corresponding binomial edge ideal in $R$. Let  $G_1,\ldots, G_c$ denote the connected components of $G$. If $R/\cJ_G$ is Cohen-Macaulay, then either $G_i$ is the complete graph or $\tau(G_i)=\frac{1}{2}$ for every $i=1,\ldots, n$.
\end{corollary}
\begin{proof}
Let  $R_i=K[x_j,y_j]_{\{i,j\}\in G}.$ 
Then, $R/\cJ_G=R_1/\cJ_{G_1}\otimes_K \ldots\otimes_K R_c/\cJ_{G_c}$. Hence, $R/\cJ_G$ is Cohen-Macaulay if and only each 
$R_i/\cJ_{G_i}$ is Cohen-Macaulay. Then, the result follows from Theorem \ref{Main Tech}.
\end{proof}

Using the previous corollary we recover a known  characterization of complete intersections for binomial edge ideals.

\begin{corollary}[{\cite[Corollary 1.2]{ChordalBEI} \& \cite[Theorem 2.2]{Rinaldo}}]
Let $G$ be a graph on $[n]$  and $\cJ_G$ its corresponding binomial edge ideal in $R$. Let $G_1,\ldots, G_c$ denote the connected components of $G$.
Then, $R/\cJ_G$ is a complete intersection if and only if $G_i$ is a path for every $i=1,\ldots, c.$
\end{corollary}
\begin{proof}
We note that $\{x_iy_j-x_jy_i:\{i,j\}\in G\}$ is a regular sequence if and only if 
$\{x_iy_j-x_jx_i:\{i,j\}\in G_i\}$ is a regular sequence for every $i=1,\ldots, c.$
Then, it suffices to show the statement assuming that $G$ is connected.
In this case, we have $\Ht(\cJ_G)=n-1.$ Then, $\cJ_G$ must be generated by $n-1$ minors. 
This means that $G$ has $n-1$ edges. Thus, $G$ must be a tree. We note that $G$ cannot have a vertex with degree greater or equal than $3$; otherwise, $\tau(G)\leq \frac{1}{3}$, which contradicts Theorem \ref{Main Tech}.
Then,  $G$ contains only vertices of degree $1$ and $2$. Hence, $G$ is a path.
\end{proof}

\subsection{Vertex-connectivity and binomial edge ideals}

We now focus on Theorem \ref{Main Depth}. We recall a property of regular  ring positive characteristic that relates cohomological dimension and depth.

\begin{proposition}[{see remark after \cite[Proposition 4.1]{P-S}}]\label{Prop PS}
If $A$ is a polynomial ring over a field positive characteristic $p,$ then $\cd(I)\leq \pd_A(A/I).$
\end{proposition}

We also need a result that relates the cohomological dimension  of the intersection of two ideals with the dimension of their sum.

\begin{proposition}[{see  \cite[Proposition 19.2.7]{BroSharp}}]\label{Prop cd dim}
Let $(A,\n,K)$ be a complete local ring. Let $\fa,\fb\subsetneq A$  be two ideals.
Suppose that $\min\{\dim A/\fa,\dim A/\fb\}>\dim A/(\fa+\fb).$ Then,
$$
\cd(\fa\cap\fb)\geq\dim A -\dim A/(\fa+\fb)-1.
$$
\end{proposition}

\begin{theorem}\label{Main Depth Char p}
Let $G$ be a connected graph on $[n]$ and $\cJ_G$ the corresponding binomial edge ideal in $R=K[x_1,\ldots,x_n,y_1,\ldots,y_n]$. 
Suppose that $\Char(K)=p>0,$ and  $G$ is not the complete graph.
If $G$ is $\ell$-vertex-connected, then $\pd(R/\cJ_G)\geq n+\ell-2.$
In particular,  $\Depth R/\cJ_G\leq n-\ell+2.$
\end{theorem}
\begin{proof}
Since $G$ is connected, we have $P_\varnothing (G)$ is a minimal prime for $\cJ_G$ by Remark \ref{MinorMinimal}. 
Let $P_\varnothing (G),Q_1,\ldots, Q_a$ denote the minimal primes of $\cJ_G.$ Let $I=Q_1\cap\ldots\cap Q_a$. We note that $\cJ_G=P_\varnothing (G)\cap I$, because $\cJ_G$ is a radical ideal.
In addition, $$\dim R/(P_\varnothing (G)+I)<\dim R/P_\varnothing (G)\hbox{ and }R/(P_\varnothing (G)+I)<\dim R/I,$$ because $P_\varnothing (G)$ and $I$ do not share any minimal primes. 

We also have that $\dim R/(P_\varnothing (G)+I)\leq n-\ell+1$ as immediate consequence of Lemma \ref{lemma dim vertex-connectivity}.
By Proposition \ref{Prop cd dim}, 
$$\cd(\cJ_G)\geq \dim R-\dim R/(P_\varnothing (G)+I)-1\geq 2n-n+\ell -1-1=n+\ell -2.$$
By Proposition \ref{Prop PS}, 
{\cb 
$\pd(R/\cJ_G)\geq \cd(\cJ_G)\geq n+\ell -2.$}
The statement regarding depth 
follows immediately from the Auslander-Buchsbaum Formula.
\end{proof}

\begin{theorem}\label{Main Depth Char Zero}
Let $G$ be a connected graph on $[n]$ and $\cJ_G$ the corresponding binomial edge ideal in $R=K[x_1,\ldots,x_n,y_1,\ldots,y_n]$. 
Suppose that $\Char(K)=0,$ and  $G$ is not the complete graph.
If $G$ $\ell$-vertex-connected, then $\pd(R/\cJ_G)\geq n+\ell-2.$
In particular, {\cb  $\Depth R/\cJ_G\leq n-\ell+2.$}
\end{theorem}
\begin{proof}
We first assume that $K=\QQ.$
Let $A=\ZZ[x_1,\ldots,x_n,y_1,\ldots,y_n]$ and $$J=(x_iy_j-x_jy_i:\{i,j\}\in G \hbox{ and }i\neq j)A.$$
Then, 
$$
\pd_{R}(R/\cJ_G)=\pd_{A\otimes_\ZZ\QQ}(A\otimes_\ZZ\QQ/J\otimes_\ZZ\QQ)=\pd_{A\otimes_\ZZ\FF_p}(A/J\otimes_\ZZ\FF_p)
$$
for $p\gg 0$ \cite[Theorem 2.3.5]{HHCharZero}. Then, the result follows from Theorem \ref{Main Depth Char p}.

The general case follows from the fact that field extensions are faithfully flat and do not change projective dimension. The claim about depth follows from the Auslander-Buchsbaum Formula.

\end{proof}
{\cb
\begin{remark}
The inequalities in Theorems \ref{Main Depth Char p} and \ref{Main Depth Char Zero} are sharp.
For instance, if $n\geq 4$ and $C_n$ is the cycle with $n$ vertices, then
$\pd(R/\cJ_{C_n})=n$ \cite[Theorem 3.8]{ZaZo}. Then,
$$
\pd(R/\cJ_{C_n})=n=n+2-2=n+\kappa(G)-2.
$$
These bounds are also sharp if  $R/\cJ_G$ is Cohen-Macaulay, because $\kappa(G)=1$ and $\Depth(R/\cJ_G)=\dim(R/\cJ_G)=n+1$ (see Proposition \ref{Prop 2-vertex-conec}).
\end{remark}
}
\section{Multiplicities of binomial edge ideals}

In this section we give an algorithmic formula to compute two important invariants associated to the singularity of a ring: the Hilbert-Samuel and Hilbert-Kunz multiplicities. 
We start by reviewing preliminaries of these invariants.
We take advantage of the fact that $\cJ_G$ is a homogeneous ideal, and then $R/\cJ_G$ is an standard $\NN$-graded $K$ algebra.

\begin{definition}
Suppose that $T$ is a standard graded $K$-algebra of dimension $d$ with maximal homogeneous ideal $\n$.
The \emph{Hilbert-Samuel multiplicity of $R$} is defined by 
$$\e(T)=\lim\limits_{r\to\infty}\frac{d! \lambda(T/\n^r)}{r^d}=\lim\limits_{r\to\infty}\frac{d! \dim_K(T/\n^r)}{r^d}.$$

If $K$ has prime characteristic $p,$ \emph{the Hilbert-Kunz multiplicity of $T$,} introduced by Monsky \cite{MonskyHK}, is defined by 
$$\ehk(T)=\lim\limits_{e\to\infty}\frac{\lambda(T/\n^{[p^e]})}{p^{ed}}=\lim\limits_{e\to\infty}\frac{\dim_K(T/\n^{[p^e]})}{p^{ed}},$$ 
where $\n^{[p^e]}=(f^{p^e}: f\in\n)T.$
\end{definition}

These two multiplicities satisfy an additivity formula. Suppose $T$ is reduced, and $\fp_1,\ldots,\fp_t$ are the minimal primes of $T$ such that $\dim T =\dim T/\fp_i.$
Then,
\begin{equation}\label{EqAdd}
\e(T)=\e(T/\fp_1)+\ldots+\e(T/\fp_t), \quad \& \quad \ehk(T)=\ehk(T/\fp_1)+\ldots+\ehk(T/\fp_t).
\end{equation}


{\cb
We  recall a fact of that allows us to compute the Hilbert-Samuel multiplicity of binomial edge ideals.
\begin{remark}\label{Prod HS}
Let $A$ and $B$ be two standard graded finitely generated $K$-algebras, and $T=A\otimes_K B$.
Let $h_A(t),h_B(t),$ and $h_T(t)$ denote the Hilbert series of $A,B$ and $T$ respectively.
There exists polynomials $f_A(t),f_B(t),$ and $f_T(t)$ such that
$$
h_A(t)=\frac{f_A(t)}{(1-t)^{\dim(A)}}, \; h_B(t)=\frac{f_B(t)}{(1-t)^{\dim(B)}},\;
\hbox{ and }\; h_T(t)=\frac{f_T(t)}{(1-t)^{\dim(T)}}.$$ Furthermore, $f_A(1)=e(A), f_B(1)=e(B),$ and $f_T(1)=e(T).$
Since $h_T(t)=h_A(t)\cdot h_B(t),$ we obtain $f_t(t)=f_A(t)\cdot f_B(t)$.
Then, $\e(T)=f_T(1)=f_A(1)f_B(1)=\e(A)\e(B).$
\end{remark}
}

We now recall the analogue of Remark \ref{Prod HS} for the Hilbert-Kunz multiplicity.

\begin{remark}\label{Prod HK}
Let $A$ and $B$ be two standard graded finitely generated $K$-algebras, where $\Char(K)=p>0$. 
Let $\m_A$ and $\m_B$ denote the the maximal homogeneous ideals of $A$ and $B$ respectively. Take $\m=\m_A\otimes_K B+A\otimes_K \m_B$.
If $T=A\otimes_K B$, then
$$
\e_{HK}(T)=\lim\limits_{e\to\infty}\frac{\dim_K (T/\m^{[p^e]})}{p^{e(\dim(A)+\dim(B))}}=\lim\limits_{e\to\infty}\frac{\dim_K A/\m^{[p^e]}_A}{p^{e\alpha}}\cdot
\lim\limits_{e\to\infty}\frac{\dim_K B/\m^{[p^e]}_B}{p^{e\gamma}}=\e_{HK}(A)\cdot \e_{HK}(B).
$$
\end{remark}

We now can give the algorithmic formula for the Hilbert-Samuel and the Hilbert-Kunz multiplicity of $R/\cJ_G$. In particular, $\e_{HK}(R/J_G)$ is a rational number, which is not true in general \cite[Theorem 8.3]{BrennerHK}.

\begin{theorem}\label{Thm Mult}
Let $G$ be a graph on $[n]$ and $\cJ_G$ the corresponding binomial edge ideal in $R=K[x_1,\ldots,x_n,y_1,\ldots,y_n].$
Let $\alpha=\max\{ c(S)-|S|: S\subseteq [n]\}$,  
$\cA=\{S\subseteq [n]:c(S)-|S|=\alpha\}$, $B_{S,1},\ldots, B_{S,j_i}$ the connected component of the graph induced by $[n]\setminus S$ for $S\in \cA$. Then, 
$$ \e(R/J_G)=\sum_{S\in \cA}\left(\prod^{j_S}_{r=1}|B_{S,r}|\right)\quad \&\quad 
\e_{HK}(R/J_G)=\sum_{S\in \cA}\left(\prod^{j_S}_{r=1}\left(\frac{|B_{S,r}|}{2}+\frac{|B_{S,r}|}{(|B_{S,r}|+1)!}\right)\right),
 $$
 where the lats equality assumes that $K$ has positive characteristic.
\end{theorem}
 \begin{proof}
By definition of $\cA$ and Theorem \ref{IntersectionPrimes}, $\{ P_S(G):S\in\cA\}$ is the set of  primes with $\dim(R/P_S(G))=\dim(R/\cJ_G)$. Then, $\cA$ is the set of minimal primes of $\cJ_G$ with maximal dimension. 
We note that $R/P_S(G)$ is the tensor product over $K$ of rings with the form
$T_\gamma =K[X]_{2\times b}/I_2(X).$ 
Since $\e(T_b)=b$ \cite[Proposition 2.15]{Bruns}, we deduce that $\e(R/P_S(G))=\prod^{j_S}_{r=1}|B_{S,r}|$  by Remark \ref{Prod HS}.
Then,
$\e(R/J_G)=\sum_{S\in \cA}\left(\prod^{j_S}_{r=1}|B_{S,r}|\right)$ by Equation \ref{EqAdd}.

Since $\e_{HK}(T_\gamma)=\frac{\gamma}{2}+\frac{\gamma}{(\gamma+1)!}$
by \cite{Eto,EtoYoshida}, we have 
 $$\e_{HK}(R/P_S(G))=\prod^{j_S}_{r=1}\left(\frac{|B_{S,r}|}{2}+\frac{|B_{S,r}|}{(|B_{S,r}|+1)!}\right)$$ 
 by Remark \ref{Prod HK}.
Then,
$\e_{HK}(R/J_G)=\sum_{S\in \cA}\left(\prod^{j_S}_{r=1}\left(\frac{|B_{S,r}|}{2}+\frac{|B_{S,r}|}{(|B_{S,r}|+1)!}\right)\right)$ by Equation \ref{EqAdd}.
\end{proof}

Using the previous formulas we compute these multiplicities in a few cases.
\begin{example}
Let $\cK_{\ell,m}$ denote a complete bipartite graph with $1\leq \ell\leq m$. For this graph, the minimal primes of its binomial edge ideal are 
$P_\varnothing(\cK_{\ell,m})$, $(x_1,\ldots, x_\ell,y_1, \ldots,y_\ell)$  if $m> 1$, and
$(x_{\ell+1},\ldots, x_{\ell+m},y_{\ell+1}\ldots,y_{\ell+m})$ if $\ell>1$ \cite[Lemma 3.2]{SZBipartite}.
Then,
$$
\e_{HK}(R/\cJ_{\cK_{\ell,m}})=
\begin{cases} 
1 & \hbox{ if }m > \ell + 1\hbox{ or }\ell = 1\hbox{ and }m > 2,\\
2m+\frac{2m}{(2m+1)!} &\hbox{ otherwise.}
\end{cases}
$$
\end{example}

\begin{example}\label{mult 1-tough}
Let $G$ be a connected graph on $[n]$ and $\cJ_G$ the corresponding binomial edge ideal in $R.$
Suppose that $G$ is a $1$-tough graph with $n$ vertices.
By Proposition \ref{equiv 1-tough}, $P_{\varnothing}(G)$ is the only minimal prime with the same dimension as $\cJ_G.$
Then, 
$\e(R/\cJ_G)=n$. If $\Char(K)=p>0,$ then
$\e_{HK}(R/\cJ_G)=\frac{n}{2}+\frac{n}{(n+1)!}$.
In particular, these values are attained when $G$ is a Halmitonian graph.
\end{example}

\section*{Acknowledgments}
We thank Craig Huneke for helpful discussions.
We are grateful to the anonymous referee for valuable comments and suggestions.
The second author gratefully acknowledges the support of the National Science Foundation for support through Grant
DMS-1502282.

\bibliographystyle{alpha}
\bibliography{References}

{\footnotesize

\noindent \small \textsc{Department of Mathematics, Purdue University, West Lafayette, IN 47907-2067.} \emph{Email address}:  {\tt banerj19@purdue.edu}

\vspace{.25cm}

\noindent \small \textsc{Department of Mathematics, University of Virginia, Charlottesville, VA  22904-4135.} \emph{Email address}:  {\tt lcn8m@virginia.edu} 
 
}

\end{document}